\documentclass[11 pt]{amsart}
\usepackage[latin1]{inputenc} 
\usepackage[pctex32]{graphics}
\usepackage{amsfonts}
\usepackage{amssymb,amscd,latexsym}
\usepackage{amsmath}
\usepackage{epsfig}
\usepackage{color}
\usepackage[all]{xy}

\newtheorem{thm}{Theorem}[section]

\newtheorem{prop}[thm]{Proposition}

\newtheorem{lema}[thm]{Lemma}
\newtheorem{rmk}[thm]{Remark}

\newtheorem{ex}[thm]{Example}

\def\OX{\mathcal{O}}
\def\Ee{\mathcal{E}}
\def\CC{\mathcal{C}}
\def\J{\mathcal{J}}
\def\Fe{\mathcal{F}}
\def\T{\mathcal{T}}

\def\K{\mathbb{K}}

\def\grass{\mathbb{G}}

\def\P{\mathbb{P}}
\newcommand{\neon}[1]{{\color{black}#1}}
\def\Hess{\operatorname{Hess}}
\def\hess{\operatorname{hess}}

\def\GOR{\operatorname{GOR}}
\def\Gor{\operatorname{Gor}}
\def\Ann{\operatorname{Ann}}
\def\Hilb{\operatorname{Hilb}}

\def\Sing{\operatorname{Sing}}

\def\s#1{\ensuremath{\operatorname{Sym}_{#1}}}

\begin{document}
\title{Developable cubics in $\P^4$ and the Lefschetz locus in $\GOR(1,5,5,1)$}
\author[Fassarella]{Thiago Fassarella}
\address{\sc Thiago Fassarella\\
Universidade Federal Fluminense, Rua Alexandre Moura 8 - S\~ao Domingos, 24210-200 Niter\'oi, Rio de Janeiro, Brasil}
\email{tfassarella@id.uff.br}

\author[Ferrer]{Viviana Ferrer}
\address{\sc Viviana Ferrer\\
Universidade Federal Fluminense, Rua Alexandre Moura 8 - S\~ao Domingos, 24210-200 Niter\'oi, Rio de Janeiro, Brasil}
\email{vferrer@id.uff.br}

\author[Gondim]{Rodrigo Gondim}
\address{\sc Rodrigo Gondim\\
Universidade Federal Rural de Pernambuco, av. Dom Manoel de Medeiros s/n, Dois Irm\~aos, 52171-900 Recife, Pernambuco, Brasil}
\email{rodrigo.gondim@ufrpe.br}

\thanks{2010 Mathematics Subject Classification. 14E05, 16E65 (primary), 14C17 (secondary).
Key words and phrases: developable cubics, Artinian Gorenstein algebra, Lefschetz property. 
The first author was partially supported by CNPq  303599/2016-2 and CAPES-COFECUB  932/19.
The third author was partially supported by IMPA summer postdoc 2019/2020 grant. }

%

\maketitle
\begin{abstract}
We provide a classification of developable cubic hypersurfaces in $\P^4$. Using the correspondence between forms of degree $3$ on $\mathbb P^4$ and Artinian Gorenstein $\K$-algebras, given by Macaulay-Matlis duality,  we describe the locus in $\GOR(1,5,5,1)$ corresponding to those algebras which satisfy the Strong Lefschetz property.  
\end{abstract}


\section{Introduction}

We work over an   algebraically closed field $\mathbb K$ of characteristic zero. The projective space over $\mathbb K$ of dimension $N$
 will be denoted by $\mathbb P^N$. In this note we will focus our attention on the classification of developable cubic hypersurfaces on $\mathbb P^4$ as well as the  Artinian Gorenstein algebras defined by them.   

An irreducible projective variety $X \subset \P^N$ is called {\it developable} if it has \neon{a} degenerate Gauss map.  
Recent progress on the classification problem of de-
velopable varieties has been made via the focal locus of the ruling
defined by fibers of the Gauss map. For instance, see \cite{AG,MT} for a classification of developable threefolds. In Section \ref{DC} we proceed with a careful analysis of the focal locus to provide a finer classification of developable cubic hypersurfaces in $\mathbb P^{4}$. Our first goal is the following  result. 

\begin{thm}\label{mainthm}
 Let $X \subset \P^4$ be an irreducible  cubic hypersurface. Assume that $X$ is not a cone.   Then $X$ is developable if and only if it is projectively equivalent to a linear section of the secant variety of the Veronese surface. 
 \end{thm}

A linear section of the secant variety of the Veronese surface is  projectively equivalent to one of the following varieties (see Section \ref{sec ver}):
\begin{enumerate}
\item  the secant variety of the rational normal quartic curve; 
\item the join of two irreducible conics sharing a single point;  this point coincides with the intersection between the planes containing the conics. 
\item the dual variety of the scroll surface $S(1,2)$. 
\end{enumerate}
Cases $(1)$ and $(2)$ correspond to cubic hypersurfaces which have nonvanishing Hessian while case $(3)$ yields a cubic with vanishing Hessian. 

To each cubic hypersurface above we can associate an Artinian Gorenstein $\K$-algebra. More generally, Macaulay-Matlis duality  offers a correspondence between forms of degree $d$ in $N+1$ variables over $\K$, not defining a cone,  and standard graded Artinian Gorenstein $\K$-algebras of socle degree $d$ and codimension $N+1$.  These algebras enjoy nice properties such as  Poincar\'e duality in cohomology theory. We give a precise definition in  Section \ref{LL}.  

Given a standard graded Artinian Gorenstein $\mathbb K$-algebra $A = \displaystyle \bigoplus_{k=0}^d A_k$, AG algebra for short,  its Hilbert vector is the vector $\Hilb(A) = (1,a_1,\ldots,a_d)$ where $a_k =\dim_{\K}A_k$. We denote by $\GOR(T)$ the space which parametrizes  AG algebras with Hilbert vector $T$. It has been extensively studied in \cite{IK}. We are interested  in algebras inside $\GOR(T)$ which satisfy the Strong Lefschetz property, SLP for short, which means that  there exists a linear form $l \in A_1$ such that every multiplication map $\mu_{l^j}:A_k \to A_{k+j}$ has maximal rank. This notion was introduced \neon{in the commutative algebra setting} by R. Stanley  and J. Watanabe, see \cite{St, W, H-W}, and was inspired by the  so called hard Lefschetz Theorem on the cohomology of smooth projective complex varieties, see for example \cite{GH}.  
The Lefschetz properties have attracted a lot of attention over the last years;
we refer to  \cite{H-W} for a survey on the area.


Now we go back to cubic hypersurfaces in $\P^4$ which are not cones. They  correspond  to AG algebras with Hilbert vector $(1,5,5,1)$, via Macaulay-Matlis duality.  
We focus on the Strong Lefschetz property  of algebras in $\GOR(1,5,5,1)$ which come from a developable cubic hypersurface. Algebras associated to cases $(1)$ and $(2)$ above have the SLP,  whereas it fails in case  $(3)$. Moreover,  by the main result of \cite{MW}, any AG algebra with Hilbert vector $(1,5,5,1)$ failing SLP comes from case $(3)$.   Sections \ref{LL} and \ref{PS} are devoted to the description of the locus of algebras in $\GOR(1,5,5,1)$ failing the SLP. \neon{The main results of these sections are summarized by the following theorem.} 
\begin{thm}\label{mainthm2} 
The space $\GOR(1,5,5,1)$ parametrizing AG algebras with Hilbert vector $(1,5,5,1) $ coincides with $\P^{34} \setminus \CC_4$,
 where $\CC_4$ is the space of cubic cones in $\P^4$. Moreover, the following assertions hold:
 \begin{enumerate}
 \item The locus $\CC_4$ is the image of a projective bundle over $\P^4$ by a birational morphism, its dimension is $23$ and its degree is $1365$. 
\item The  locus of algebras failing SLP coincides with $\mathcal K \setminus \mathcal{C}_4 $,  
 where $\mathcal K $
is a  rational projective variety of dimension $18$ and degree $29960$. More precisely, $\mathcal K $  is the image of  a projective bundle over the Grassmannian $\mathbb{G}(2,4)$ by a birational morphism. 
 \item The intersection $\mathcal K \cap \CC_4 $ is a divisor in $\mathcal K$ of degree $116420$.
 \end{enumerate}
\end{thm}

We note that the locus of algebras in $\GOR(1,5,5,1)$ failing SLP coincides with the locus of algebras with Jordan type $4^1 \oplus 2^{3} \oplus 1^2$, while any other algebra has Jordan type  $4^1 \oplus 2^{4}$. In particular $4^1 \oplus 2^{3} \oplus 1^2$ is the only possible degeneration of the general Jordan type. 
This phenomenon cannot occur for
 $N \leq 3$: in this situation having vanishing Hessian is equivalent to be a cone. Hence, \neon{for $N \leq 3$}, 
 any algebra in $\GOR(1,N+1,N+1,1)$ has the SLP, see \cite{CG}.
 
%
  
\section{Developable cubics in $\mathbb P^4$}\label{DC}

\subsection{Basic definitions}

Given a rational map $\varphi:X\dashrightarrow Y$ between projective varieties,  its image is the closure of $\varphi(U)$ in $Y$, where $U$ is the maximal domain where $\varphi$ is defined. \neon{We say that $\varphi:X\dashrightarrow Y$ is dominant if it has $Y$ as image.}

 Let $X\subset \mathbb P^{N}$ be a projective subvariety  of dimension $n\ge 1$. Let $(\mathbb P^N)^*$ denote the space of hyperplanes in $\mathbb P^N$. We denote by ${\rm Con}X\subset \mathbb P^{N}\times (\mathbb P^{N})^{*}$ the {\it conormal variety} of $X$: this is the closure of the set of pairs $(x,H)$ such that $x$ is a regular point of $X$ and $H$ contains the tangent space $T_{x}X$. 
Let $X^{*}$ be the image of the  projection in the second coordinate. It is the {\it dual variety} of $X$. Given a point $x\in \mathbb P^{N}$, we define $x^{*}\subset (\mathbb P^N)^*$ as the set of hyperplanes passing through it. 

Let $\mathbb G(n,N)$ denote the Grassmannian of $n$-planes in $\mathbb P^N$. The {\it Gauss map} $\gamma:X\dashrightarrow \mathbb G(n,N)$ associates to each regular point $x\in X$ the tangent space $T_{x}X\in \mathbb G(n,N)$.  We denote by $X^{\lor}$ the image of $\gamma$.  We say that an irreducible projective variety $X$ is {\it developable} if ${\rm dim}X^{\lor}< n.$


We are particularly interested in the case where $X$ is a hypersurface. Assume that it is the zero locus $X=V(f)$ of a non-constant homogenous polynomial $f$ in $N+1$ variables.  Its {\it polar map} is the rational map
\begin{eqnarray*}
 \Phi_f:  \P^N  &\dashrightarrow&   (\P^N)^* \\
            p &\mapsto& \left(f_0(p): f_1(p):  ...: f_N(p)\right)
\end{eqnarray*}
%
where $f_i$ is the partial derivative of $f$ with respect to $x_i$. We denote by $Z $ the image of the polar map, called {\it polar image} of $X$. The restriction of the polar map to $X$ is just the Gauss map $\gamma:X\dashrightarrow (\mathbb P^{N})^*$, and $X^{\vee}$ coincides with $X^*$.  We note that since we are working in characteristic zero, the Reflexivity Theorem  says that $(X^*)^*=X$, see \cite[p. 208]{Ha} for an elementary proof. 



Let us denote by $\operatorname{Hess}_f $  the {\it Hessian matrix} of $f$, namely the matrix of the
second derivatives. Its 
determinant is the {\it Hessian determinant}. We shall say that $X=V(f)$ or $f$ has {\it vanishing Hessian},  if its Hessian determinant is null.  
Therefore $\Phi_{f}$ is nondominant if and only if $f$ has vanishing Hessian. This is equivalent to say that the derivatives  $f_0, \dots, f_N$    of $f$ are algebraically dependent. 
We summarize the above discussion in the following proposition. 

\begin{prop}\label{prop.hess_vs_cone}
  Let $X = V(f) \subset \P^N$ be a hypersurface and $Z$ its polar image. The following conditions are equivalent. 
  \begin{enumerate}
    \item $X$ has vanishing Hessian;
    \item The partial derivatives of $f$ are algebraically dependent;
    \item $Z$ is a proper subvariety of $(\P^{N})^*$.
  \end{enumerate}
\end{prop}

The singular locus and the polar image of a hypersurface with vanishing Hessian have a relevant role.  
The following proposition gives a relation between them. Its proof can be 
found in the original work of Perazzo, see \cite{Pe} for the cubic case and \cite[p. 21]{ZakHesse} for any degree. 
\begin{prop}[\cite{Pe, ZakHesse}] \label{prop:zdualiny}
  Let $X \subset \P^N$ be a hypersurface with vanishing Hessian. Then $Z^* \subset \Sing(X)$.
  \end{prop}

%
  
\begin{rmk}\rm Hypersurfaces with vanishing Hessian 
are developable. To see this, we assume that $X$ has vanishing Hessian. First we note that $X^*$ is a proper subvariety of $Z$. In fact if $X^*=Z$ then $Z^*=X$,  but this contradicts Proposition \ref{prop:zdualiny}. 
The strict inclusions of irreducible varieties $X^* \subsetneq Z \subsetneq (\P^N)^*$ imply that $\dim X^* < N-1$, hence 
$X$ is developable.
\end{rmk}

Given projective subvarieties $V,W\subset \mathbb P^N$, we denote by $S(V,W)$ the {\it join} between them. It is the closure of the union of lines in $\mathbb P^N$ joining $V$ to $W$. In particular $S(V)=S(V,V)$ is the secant variety of $V$. \neon{A subvariety $V\subset \mathbb P^{N}$ is a cone if there exists $x\in V$ such that $S(x,V)=V$.}  This motivates the definition of the vertex of $V$
$$
{\rm Vert}(V)=\{x\in V \;:\; S(x,V)=V\}.
$$


Cones are the simplest examples of hypersurfaces with vanishing Hessian. Now we state the following useful proposition the proof of which will be left
to the reader.

\begin{prop}{\label{prop:equivalent_definitions_of_cones}}
Let $X = V(f) \subset \P^N$ be a hypersurface. Then the following conditions are equivalent:
\begin{enumerate}
  \item[(i)] $X$ is a cone;
  \item[(ii)] The partial derivatives of $f$ are linearly dependent;
  \item[(iii)]  $Z$ is contained in a hyperplane of $(\P^N)^*$;
  \item[(iv)] $X^{*}$  is contained in a hyperplane of $(\P^N)^*$;
  \item[(v)]  Up to a projective transformation $f$ depends on at most $N$ variables.
\end{enumerate}
\end{prop}

There are many classical examples of varieties with vanishing Hessian   which are not cones.  The following example  appears in the work of Gordan and Noether \cite{GN} and Perazzo \cite{Pe}, called 
{\it un esempio semplicissimo}. 
\begin{ex} \label{simples}\rm
Let $X = V(f) \subset \P^4$ be the irreducible hypersurface given by 
$$f = x_0x_3^2+x_1x_3x_4+x_2x_4^2.$$
We can check that $X$ is not a cone, showing for example the linear independence between the partial derivatives.  But since $f_0f_2=f_1^2$ is an algebraic relation among them, $X$ has vanishing Hessian. 
\end{ex}

\subsection{Linearity of general fibers and focal locus}

{\neon{In this section we recall some useful results about developable varieties. The main basic fact is that the general  fiber of the Gauss map  is a union of finitely many linear spaces, see \cite[p. 95]{C.Segre2}. Actually,  it has been proved by Zak that the closure of a general fiber is irreducible. For instance, see  \cite[Theorem 2.3]{Zak1} or \cite[p. 87]{FP}.}} 

\begin{thm}[\cite{C.Segre2, Zak1}]\label{linearidade}
Let $X\subset \mathbb P^{N}$ be  an irreducible projective variety. If $X$ is   developable then the closure of a general fiber of $\gamma$ is a 
 linear subspace. 
\end{thm}

\neon{Let $X\subset \mathbb P^{N}$ be an irreducible projective variety of dimension $n\ge 1$.} When $X$ is a hypersurface with vanishing Hessian, the fibers of its  polar map share this phenomenon of linearity. We state this result below, the proof can be found in \cite[Proposition 4.9]{ZakHesse}.  

\begin{thm}[\cite{ZakHesse}]\label{linearidadepolar}
Let $X\subset \mathbb P^{N}$ be a reduced hypersurface with vanishing Hessian. The closure of the fiber of $\Phi_{f}$ over a general point $z\in Z$ is a union of finitely many linear subspaces passing through the subspace $(T_{z}Z)^{*}$.  
\end{thm}

\neon{In the next example we illustrate the linearity of fibers of $\gamma$ and $\Phi_{f}$. Before moving on, we  recall the definition of  a scroll surface $S(1,2)\subset \mathbb P^4$. Let  $l\subset \mathbb P^4$ be a line and let  $C\subset \mathbb P^4$ be a conic lying on a plane complementary to $l$.  Given an isomorphism $\varphi: l\longrightarrow C$ we let $S(1,2)$ be the union of the lines $\overline{x,\varphi(x)}$ joining points of $l$ to corresponding points of $C$. We note that all the scrolls $S(1,2)$ are projectively equivalent,  see for example \cite[Example 8.17]{Ha}.}

\begin{ex}\rm\label{novamente}
We want to describe the fibers of $\gamma$ and $\Phi_{f}$ where $f = x_0x_3^2+x_1x_3x_4+x_2x_4^2$. In particular, we will see in this example that $X^{*}\subset (\mathbb P^4)^*$ is a scroll surface $S(1,2)$.  See Figure \ref{scroll}.

The singular locus of $X=V(f)$, with reduced structure, is $Y=V(x_{3},x_{4})$. Its polar image is the quadratic cone $Z=V(y_{0}y_{2}-y_{1}^{2})\subset (\mathbb P^{4})^{*}$ which has as vertex the line $l=Y^{*}$. Therefore $Z^{*}$ is a conic contained in $Y$. 
Consider the plane $P=V(y_{3},y_{4})\subset (\mathbb P^{4})^{*}$, observe that $C=Z\cap P$ is a conic and  $Z$ is the join between $C$ and $l$.

We denote by $\mathbb{P}_{t}^{3}$, $t\in \mathbb{P}^1$ the family of hyperplanes containing $Y$ and 
for each $t\in \mathbb{P}^1$ 
let  $\eta_t\in (\mathbb{P}^4)^*$ be the corresponding point of $l$. The reader can check that  $\mathbb P_{t}^{3}\cap X$ is a union of a plane $\mathbb P_{t}^{2}$ and $Y$, where $Y$  appears with multiplicity two.  A direct calculation shows that for a general point $x\in \mathbb P_{t}^{3}$ the closure of $\Phi_{f}^{-1}(y)$, $y=\Phi_{f}(x)$, is a line contained in $\mathbb P_{t}^{3}$ and passing through $\xi_{t}=(T_{y}Z)^{*}\in Z^{*}$. In particular, for a general $x\in \mathbb P_{t}^{2}$  the closure of $\gamma^{-1}(y)$ is a line contained in $\mathbb P_{t}^{2}$ passing through $\xi_{t}$. Hence $X$ is swept out by planes $\mathbb P_{t}^{2}$, and  fibers of $\gamma$ lying  in $\mathbb P_{t}^{2}$  determine a star of lines passing through the point $\xi_{t}\in Z^{*}$.

Now we prove that $X^{*}$ is a scroll surface $S(1,2)$. For a general point $x\in \mathbb P_{t}^{2}$, the tangent space $T_{x}X$ contains $\mathbb P_{t}^{2}$. Therefore the image of $\mathbb P_{t}^{2}$ by $\gamma$ is $(\mathbb P_{t}^{2})^{*}\cong l_t'$. 
Let $\mu_t\in (\mathbb{P}^4)^*$ be the point corresponding to the unique hyperplane $H_t$ containing $\mathbb P_{t}^{2}$ and $P^{*}$ (as $P^*\subset H_t$, $\mu_t\in Z\cap P=C$). Observe that  $l_t'$ is the line passing through $\eta_t\in l$ and through $\mu_t\in C$ . This shows that $X^{*}$ is a scroll $S(1,2)$  which has as rulings the lines passing through $\eta_{t}\in Y^{*}=l$ and $\mu_{t}\in C$, $t\in \mathbb P^{1}$.
\end{ex}

\begin{center}
\begin{figure}[h]
\centering
\includegraphics[height=1.5in]{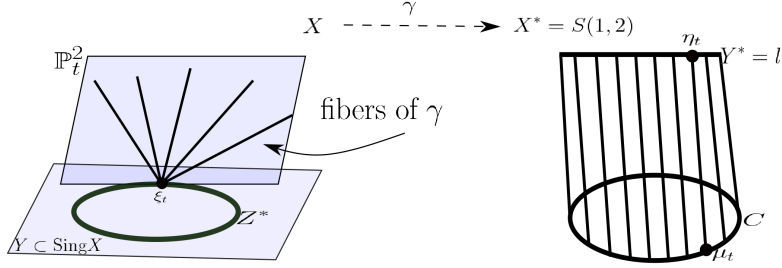}
\caption{Cubic hypersuface with vanishing Hessian.}
\label{scroll}
\end{figure}
\end{center}

\bigskip

\label{focus} Let $X\subset \mathbb P^{N}$ be a developable projective variety. By Theorem \ref{linearidade}, $X$ is ruled by linear subspaces (fibers of the Gauss map) of dimension $k$, where  
$$k={\rm dim}(X) - {\rm dim}(X^{\lor}).$$  
Let $U\subset X$ be the open subset where $\gamma$ has maximal rank. For each $x\in U$, let $L_{x}$ be the $k$-dimensional subspace passing through $x$ such that $\gamma$ is constant along $L_{x}$. 

We will denote by $B_{\gamma}$ the closure in $\mathbb G(k,N)$ of the set  $\{L_{x}\;:\;x\in U\}$. We shall say that $B_{\gamma}$ is the family of $k$--dimensional subspaces determined by fibers of $\gamma$.  Let $B'_{\gamma}$ be a  desingularization of $B_{\gamma}$ and $\mathcal I\subset B_{\gamma}'\times \mathbb P^{N}$ the incidence variety of $B_{\gamma}'$ with natural projection
$
\psi:\mathcal I\longrightarrow X.
$ 
For a general $x\in X$ the fiber $\psi^{-1}(x)$ coincides with the point $(L_{x}, x)\in \mathcal I$. 

Let $R_{\psi}$ be the ramification divisor of $\psi$ and $\pi:\mathcal I \longrightarrow B_{\gamma}'$ the natural projection on the first coordinate. We can write $R_{\psi}=H_{\psi}+V_{\psi}$ where the restriction of $\pi$ to any irreducible component of the support of $H_{\psi}$ is dominant and of the support of $V_{\psi}$ is nondominant. We say that $H_{\psi}$ is the {\it horizontal divisor} and $V_{\psi}$ is the {\it vertical divisor}. The direct image  by $\psi$ of the horizontal divisor,  denoted by $\Delta=\psi_{*}(H_{\psi})$, is called the {\it focal locus} of $X$.

We note that the restriction of $\psi$ to a general fiber of $\pi$ 
$$
\psi|_{\pi^{-1}(L)}:\pi^{-1}(L)\longrightarrow L
$$ 
is an isomorphism.  So the restriction of $H_{\psi}$ to $\pi^{-1}(L)$ defines a divisor in $L$  which coincides with the restriction of the focal locus of $X$ to $L$. This divisor will be denoted by $\Delta_{L}$.  

One of the main results concerning developable varieties is the following. For the proof see \cite[Theorem 3.4.2]{IL}.

\begin{thm}\label{focal locus}
Let $X\subset\mathbb P^{N}$ be a developable projective variety. If $X$ is not a linear subspace, then $X$ is singular and its focal locus is contained in ${\rm Sing}(X)$. Moreover, for a general $L$ belonging to $B_{\gamma}$, the restriction of the focal locus to $L$ is a divisor $\Delta_{L}$ in $L$ of degree ${\rm dim}(B_{\gamma})$.
\end{thm}

\subsection{Cubics with vanishing Hessian}\label{vanishing}
Revisiting the work of Perazzo \cite{Pe}, in \cite{GRu} the authors provide a classification of cubic hypersurfaces with vanishing Hessian in $\mathbb P^{N}$, for $N\le 6$.  In this section we rebuild the classification for $N=4$. This digression will be useful in the next section. 



\begin{lema} \label{lemma:maximal_singular_locus}
 Let $X \subset \P^N$, $N\ge 3$,  be an irreducible cubic hypersurface. Assume there is a component $Y$ of ${\rm Sing}(X)$, with $\dim Y = \dim X -1$. Then $Y$ is a linear subspace.
\end{lema}

\begin{proof}
 Since $X$ has degree 3, the secant variety $S(Y)$ of $Y$ must be contained in $X$. 
Hence either $S(Y)= Y$, in this case $Y$ is a linear subspace, or $S(Y)=X$. But the 
second case cannot occur, because  the equality $$\dim S(Y) = \dim Y + 1$$  implies that $X=S(Y)$ is a linear subspace, see \cite[Proposition 1.2.2]{Ru}.
\end{proof}

\begin{rmk}\rm
\neon{Here, we note that an irreducible developable cubic surface $X\subset \mathbb P^N$ is a cone.  It is well known that  a developable surface $X\subset \mathbb P^N$ must be either a cone or the tangent developable to a curve, see \cite[Theorem 3.4.6]{IL}. In the last case the curve lies in the singular locus of $X$. Consequently, it follows from Lemma \ref{lemma:maximal_singular_locus} that any developable  cubic surface in $\mathbb P^{3}$ is a cone. Now, let $N\ge 3$. By taking the smaller linear subspace containing $X$, we can assume that $X\subset \mathbb P^N$ is non-degenerate, and let $H\simeq \mathbb P^{N-1}$ be a general hyperplane. We will show that $N=3$. Since $C = X\cap H$ is an irreducible non-degenerate curve in $H$ of degree $3$, we conclude that $N\le 4$, see \cite[Proposition 18.9]{Ha}. If $N=4$, then  $C$ is the rational normal curve in $H$ and  $X$ is smooth, but since developable (non-linear) varieties are singular, see [Theorem \ref{focal locus}],  this leads us to a contradiction.} 
\end{rmk}

\begin{lema}\label{idealquadratico}
Let $X \subset \P^4$ be an irreducible cubic hypersurface. Assume that $X$ is not a cone. If ${\rm Sing}(X)$ contains a linearly embedded $\mathbb P^{2}$ then $X$ is projectively equivalent to $V(f)$, where  $f=x_0x_3^2 + x_1x_3x_4 + x_2x_4^2$.
\end{lema}

\begin{proof}
Let us suppose $W=V(x_{3},x_{4})\subset  {\rm Sing}(X)$. If $f$ is an irreducible polynomial defining $X$ one can write $f=ax_{3}+bx_{4}$, where $a$ and $b$ are polynomials of degree two. Since the derivatives of $f$ must vanish in $W$ we can write $$f=l_{0}x_{3}^{2}+l_{1}x_{3}x_{4}+l_{2}x_{4}^{2}$$ where $l_{i}$, $i=0,1,2$,  are linear forms. If $X$ is not a cone, then $l_{0}$, $l_{1}$ and $l_{2}$ are linearly independent, so there is a projective transformation such that $$f=x_0x_3^2 + x_1x_3x_4 + x_2x_4^2.$$
\end{proof}


The following proposition is a classical result of Perazzo \cite{Pe}.  

\begin{prop}\label{classhess}
Let $X \subset \P^4$ be a cubic hypersurface. Assume that $X$ is  not a cone. The following conditions are equivalent:
\end{prop}
\begin{enumerate}
\item[(i)] $X$ has vanishing Hessian; 
\item[(ii)] $X^*$ is projectively equivalent to the scroll surface $S(1,2)$.
\end{enumerate}

\begin{proof}
First we will show that we can assume $X$ irreducible. Suppose that  $X=V(f)$ is reducible and  is not a cone, then $f=hq$ where $h$ is a linear form and $q$ is homogeneous of degree $2$. In this case its polar map $\Phi_f$ is dominant, as follows  by a straightforward computation. This also can be proved (when $\mathbb K=\mathbb C$) by using the following identity 
$$
d(X) = d(V(h)) + d(V(q)) + d(V(h)\cap V(q))
$$
where $d(V)$ denotes the degree of the polar map associated to $V$, see \cite[Corollary 4.3]{FM}.  Since $V(h)\cap V(q)$ is a smooth conic, recall that we are assuming $X$ is not a cone, then the right side of the identity is positive, which implies that $\Phi_f$ is dominant.  

Now we suppose that $X$ is irreducible and has vanishing Hessian. By Proposition \ref{prop:zdualiny}, we get $Z^*\subset \Sing X$. 
We can assume $\dim(Z^{*})\ge 1$, otherwise $Z$ is contained in a hyperplane and Proposition \ref{prop:equivalent_definitions_of_cones} ensures that $X$ is a cone. 

We will show that $Z^{*}$ cannot be a component of  ${\rm Sing}(X)$. Let us consider the Perazzo map
\begin{eqnarray*}
\mathcal P_{X}: \mathbb P^{N} &\dashrightarrow& \mathbb G({\rm codim} Z -1,N) \\
                                  x &\mapsto& (T_{\Phi_{f}(x)}Z)^{*}.
\end{eqnarray*} 
Since $X$ is an irreducible cubic hypersurface, the closure of a general fiber of $\mathcal P_{X}$ is a linear space, see \cite[Theorem 2.5]{GRu}. According with \cite[Proposition 2.16]{GRu} this implies that $Z^{*}$ lies in the intersection of fibers of $\mathcal P_{X}$. And from \cite[Proposition 2.13]{GRu} this is equivalent to say that the linear span $<Z^{*}>$ lies in ${\rm Sing}(X)$. Finally, this ensures that if $Z^{*}$ is a component of ${\rm Sing}(X)$ then $Z^{*}=<Z^{*}>$, which implies that $Z^{*}$ is a linear subspace. But, in this case $X$ must be a cone.

So far we have proved  that $\dim Z^*\ge 1$ and $Z^*$ cannot be component of $\Sing X$. Hence, one may assume that ${\rm Sing}(X)$ contains a two--dimensional component. It follows from  Lemma \ref{lemma:maximal_singular_locus} and Lemma \ref{idealquadratico} that $X$ is projectively equivalent to $V(f)$, where  $f=x_0x_3^2 + x_1x_3x_4 + x_2x_4^2$. This is enough to conclude that $X^* \simeq S(1,2)$, see Example \ref{novamente}.

The converse is immediate. Since all scrolls $S(1,2)$ are projectively equivalent,  $X^* \simeq S(1,2)$ implies that $X$ is projectively equivalent to $V(f)$, $f= x_0x_3^2 + x_1x_3x_4 + x_2x_4^2$.

\end{proof}


\subsection{Sections of the secant variety of the Veronese surface}\label{sec ver}
\neon{Linear sections of the secant variety of the Veronese surface are examples of developable cubic hypersurfaces in $\mathbb P^4$;   there are three classes of them, up to projectivity, which we now describe.}

\neon{Let $V$ be the three-dimensional vector space $\mathbb K^3$. The  algebraic group ${\rm GL}(3)$ acts on $V$ and consequently induces an action of   ${\rm PGL}(3)$ on $\mathbb P^5 = \mathbb P({\rm Sym}_2V)$. We can also consider the action of ${\rm PGL}(3)$ on the dual space   $(\mathbb P^5)^* = \mathbb P({\rm Sym}_2V^*)$ which we see as the space of conics in $\mathbb P^2$. 
The image $\mathcal V$ of the Veronese map 
\begin{eqnarray*}
 v:  \mathbb P(V) &\longrightarrow&  \mathbb P({\rm Sym}_2V)\\
              l &\mapsto& l^2
\end{eqnarray*}
is the Veronese surface.  %
Using coordinates $(a:b:c)$ of $\mathbb P^2$, the map above can be given by
\[
(a:b:c) \mapsto (a^2: b^2: c^2: ab: ac: bc)
\]
and if we  identify $\mathbb P^5$ with the projectivization of the space of symmetric matrices, by sending each point $(x_0:\cdots :x_5)$ of $\mathbb P^5$ to 
\[
\begin{pmatrix}
	x_0 & x_3 & x_4\\ 
	x_3 & x_1 & x_5 \\
	x_4 & x_5 & x_2 
	\end{pmatrix}
\]
then $\mathcal V$ coincides with the locus of matrices of rank one. We note that its secant variety $S(\mathcal V)$ is a cubic hypersurface. Indeed, since any singular matrix can be written as a sum of two rank one matrices then we see that  $S(\mathcal V)$ corresponds to the locus of singular matrices.}

%

\neon{We now  describe  $\mathcal V^*$ and the linear sections of $\mathcal V$.} The preimage by $v$ of a hyperplane in $\mathbb P^5$ is a conic in $\mathbb P^2$. The hyperplane is tangent to $\mathcal V$ if and only if the conic is singular, thus the dual variety \neon{$\mathcal V^*\subset (\mathbb P^5)^*$} of $\mathcal V$ is isomorphic to the locus of singular conics \neon{in $\mathbb P^2$}. \neon{Since this last coincides with the locus of singular symmetric matrices, we conclude that $\mathcal V^*\simeq S(\mathcal V)$.} A pair of distinct lines \neon{in $\mathbb P^2$} \neon{comes from}  a \neon{hyperplane} which is tangent to  $\mathcal V$  at a single point and whose intersection with \neon{$\mathcal V$} yields a pair of conics sharing a single point. A double line \neon{in $\mathbb P^2$} corresponds to a \neon{hyperplane} which is tangent along a conic.  


\neon{Since $\mathcal V^*\simeq S(\mathcal V)$  then by biduality we conclude that  $S(\mathcal V)^*\simeq \mathcal V$. In addition,  the  action of ${\rm PGL}(3)$ on $(\mathbb P^5)^*$ gives three orbits: 
\begin{itemize}
\item $\mathcal U_1=(\mathbb P^5)^*\backslash \neon{\mathcal V^*}$,  yielding  sections which are transverse to $\mathcal V$;
\item $\mathcal U_2=\neon{\mathcal V^*\backslash S(\mathcal V)^*}$, corresponding to sections which are tangent to $\mathcal V$ at a single point; and 
\item the closed orbit \neon{$S(\mathcal V)^*$}, giving sections which are tangent to \neon{$\mathcal V$} along a conic. 
\end{itemize}
}
\neon{Given $H\in (\mathbb P^5)^*$ we set $X=H\cap S(\mathcal V)$ for the corresponding linear section of $S(\mathcal V)$.  From the discussion above,  we obtain the following possibilities: 
\begin{enumerate}
\item if $H\in \mathcal U_1$ then $H\cap \mathcal V$ is a rational normal curve in $\mathbb P^4$ and  $X$ is  the secant variety of this curve;
\item  if $H\in \mathcal U_2$ then $H\cap \mathcal V$ is a union of two irreducible conics $C_1$ and $C_2$ sharing a single point,  and $X$ is projectively equivalent to the join $S(C_1,C_2)$ between them;
\item if $H\in \neon{S(\mathcal V)^*}$ then $X$ has vanishing Hessian and it is projectively equivalent to $V(f)$, where  $f=x_0x_3^2 + x_1x_3x_4 + x_2x_4^2$.
\end{enumerate} 
}

\neon{Our next goal is  to prove Theorem \ref{mainthm},  for this purpose we will need some preliminary lemmas.} 

\subsection{Preliminary Lemmas}

Let $B_{\gamma}$ be the family of linear subspaces determined by fibers of the Gauss map $\gamma$. If the dual variety $X^{*}$ of $X$ has dimension two, then $B_{\gamma}$ is a two--dimensional family of lines. Theorem \ref{focal locus} ensures that the restriction of the focal locus $\Delta$ of $X$ to a general line $L$ belonging to $B_{\gamma}$ is a divisor $\Delta_{L}$ of degree two in $L$. If $X^{*}$ has dimension one then $B_{\gamma}$ is a $1$-dimensional family of $2$-linear subspaces. Applying Theorem \ref{focal locus} again, we see that $\Delta_{L}$ is a divisor of degree one in $L$. The next results will be useful in the proof of Theorem \ref{mainthm}.


\begin{lema}\label{dim two}
Let $X\subset \mathbb P^{4}$ be an irreducible developable cubic hypersurface with nonvanishing Hessian, then $X^{*}$ has dimension $2$.
\end{lema}

\begin{proof}
Since $X$ is not a linear subspace, then the dimension of $X^{*}$ is at least one. We will show that if $X^*$ has dimension $1$ then $X$ has vanishing Hessian. Assume that $\dim X^*=1$. For a general element $L\in B_{\gamma}$, $\Delta_{L}$ is a divisor of degree one in $L$ and  $\Delta$ is contained is ${\rm Sing}(X)$. If  $\Delta_{L}\simeq \mathbb P^{1}$ varies with $L$, then the dimension of the singular set of $X$ is at least two. Lemma \ref{lemma:maximal_singular_locus} and Lemma \ref{idealquadratico} will imply that $X$ has vanishing Hessian. If $\Delta_{L}$ is a fixed line, say $l\simeq\mathbb P^{1}$, when $L$ varies in $B_{\gamma}$ then we will show that $X$ must be a cone \neon{whose} vertex contains $l$. This contradicts our hypothesis. Given $y\in l$, let $z\in S(y,X)$ general: $z\in <y,x>$ for general $x\in X$.  We are assuming that the linear subspace $L_{x}\in B_{\gamma}$ passing trough $x$ contains $l$. In particular $<y,x>\subset L_{x}\subset X$.  This implies that $z\in X$. Since $z\in S(y,X)$ is general, we get $S(y,X)\subset X$. This shows that $S(y,X)=X$ and consequently $y\in {\rm Vert}(X)$, which is a contradiction.  Therefore $X^{*}$ has dimension $2$ and this concludes the proof.
\end{proof}

\begin{lema}\label{2p}
Let $X\subset \mathbb P^{4}$ be an irreducible developable cubic hypersurface such that the support of the focal locus $\Delta$ is an irreducible curve. Besides, assume that  $X^{*}$ has dimension two. If the restriction of $\Delta$ to a general line belonging to $B_{\gamma}$ is one point of multiplicity two, then $X$ has vanishing Hessian. 
\end{lema}

\begin{proof}
Let us denote by $C$ the support of $\Delta$. Given a general point $x\in C$, let $V_{x}$ be the  cone determined by lines of $B_{\gamma}$ passing through it. Notice that $X$ is a union of these $V_{x}$ when $x$ varies in $C$.  The developable hypothesis on $X$ implies that these cones are tangent planes to $C$, see \cite[p. 454]{MT}. Therefore $C$ cannot be a line, otherwise $X$ would be a cone. 

Since $X$ has degree three, the secant variety of $C$, $S(C)$ is contained in $X$.  Now we analyze the dimension of $S(C)$. If it has dimension three, then $S(C)=X$ and by Terracini Lemma  $\Delta_L$ coincides with two distinct points, this contradicts our hypothesis on $\Delta_{L}$. If the dimension of $S(C)$  equals  two, we will show that $S(C)$ is contained in the singular set of $X$, which implies that $X$ has vanishing Hessian. So, assume that
$$
{\rm dim}(S(C))=2={\rm dim}(C)+1.
$$
In this situation, $S(C)$ is a linearly embedded $\mathbb P^{2}$, see \cite[Proposition 1.2.2]{Ru}.   Suppose that  $S(C)\simeq \mathbb P^{2}$ is not contained in ${\rm Sing}(X)$. Let $q\in S(C)$ be a smooth point of $X$ and take a tangent line $l_{x}$ of $C$ at $x\in C$ passing through $q$. Since $q$ belongs to the plane $V_{x}$, the tangent space $T_{q}X$ must contain $V_{x}$. Thus, $T_{q}X$ is the join between $S(C)$ and $V_{x}$, that is 
$$
T_{q}X=S(S(C),V_{x}).
$$
Hence, the tangent space of $X$ is constant along $l_{x}$.  But, two lines $l_{x}$ and $l_{x'}$, for distinct points $x$ and $x'$, must intersect at one point. Therefore the tangent space of $X$ is constant along $S(C)$. If $H$ denotes the tangent space at one general point $q\in S(C)$, then we have
$
V_{x}\subset H
$
for a general point $x\in C$. In this case, we must have $H=X$ and this contradicts our hypothesis $\deg X = 3$.   

Therefore $S(C)\simeq \mathbb P^{2}$ is contained in ${\rm Sing}(X)$.  Lemma \ref{lemma:maximal_singular_locus} and Lemma \ref{idealquadratico} imply that $X$ has vanishing Hessian.
\end{proof}

In order to prove Theorem \ref{mainthm} we also need the following result. 

\begin{lema}\label{secQRN}
Let $C\subset \P^4$ be a non-degenerate irreducible curve whose secant variety,  $S(C)$,  is a cubic hypersurface. Then $C$ is a rational  normal quartic curve. 
\end{lema}

\proof
It is enough to show that $\deg C = 4$. Let $x\in C$ be a smooth point, $L=T_xC$ the tangent line at $x$ and $P=\mathbb P^2\subset\mathbb P^4$ a linear space skew to $L$, that is, $P\cap L=\emptyset$. We consider the projection $\pi:C \dashrightarrow P$ from $L$ which sends $y\in C\backslash (C\cap L)$ to 
$$
\pi(y) = <L,y>\cap P.
$$
Note that  $\tilde{C} = S(L,C)\cap P$ is the closure of the image of $C$ by $\pi$. We will show that $\tilde{C}$ has degree 2 and this implies that $C$ has degree 4.  

Let $\frak C_x$ be the tangent cone of $S(C)$ at $x$. It has $S(L,C)$ as an irreducible component, see \cite[Theorem 3.1]{CR}. Since $\deg S(C) = 3$, we get $\deg \frak C_x = 2$. But $\frak C_x$ cannot be decomposed  as product of hyperplanes, otherwise $C$ would be degenerate. This shows that $S(L,C) = \frak C_x$ and therefore $\tilde{C}$ has degree 2. This concludes the proof. 
\endproof

\subsection{Proof of Theorem \ref{mainthm}}

Assume that  $X \subset \P^4$ is a developable  irreducible  cubic hypersurface which is not a cone.  If $X$ has vanishing Hessian, then Proposition \ref{classhess} yields $X^* \simeq S(1,2)$. Therefore $X$ corresponds to a section of $S(\mathcal{V})$ which is tangent to $\mathcal{V}$ along a conic (see Section \ref{sec ver}).  
 
 Suppose that $X$ has nonvanishing Hessian. By Lemma \ref{dim two} we get that $X^{*}$ has dimension two. Thus, $B_{\gamma}$ is a two--dimensional family of lines. As consequence of Theorem \ref{focal locus}, the restriction of the focal locus to a general line $L\in B_{\gamma}$ is a divisor of degree two in $L$. 

We first remark that any irreducible component of the support $|\Delta|$ of $\Delta$ has dimension one. In fact, if there exists a zero dimensional component, say $x_{0}\in X$, then $X$ must be a cone because every line $L\in B_{\gamma}$ must pass through $x_{0}$. Besides that, since $\Delta\subset {\rm Sing}(X)$ then from Lemma \ref{lemma:maximal_singular_locus} and Lemma \ref{idealquadratico} the existence of a two--dimensional component of $|\Delta|$ will imply that $X$ has vanishing Hessian. 

The focal locus $\Delta$ is the direct image of the horizontal divisor $H_{\psi}$. Recall  that the restriction of $\psi:\mathcal I \longrightarrow X$ to $\pi^{-1}(L)$ gives an isomorphism $\pi^{-1}(L)\simeq L$, for general $L$. The restriction of $H_{\psi}$ to $\pi^{-1}(L)$  is a divisor of degree two which corresponds to $\Delta_{L}$, via this isomorphism. Therefore, the support of $H_{\psi}$ has at most two irreducible components. A fortiori, the number of irreducible components of $|\Delta|$ is at most two. 

We will see that if  $|\Delta|$ has two irreducible components then it is a linear section of $S(\mathcal V)$, $X=H\cap S(\mathcal V)$ with $H\in \mathcal V^*\setminus S(\mathcal V)^*$. Suppose  $|\Delta|$ is a union of two distinct irreducible curves, say $C_{1}$ and $C_{2}$. Hence $X$ must be the join between them, $X=S(C_{1},C_{2})$. We first remark that $C_{1}$ and $C_{2}$ are plane curves. Indeed, if for example $S(C_{1})$ has dimension $3$ then $X=S(C_{1})$ and $|\Delta|=C_1$,  contradicting our hypothesis on $\Delta$.  If $C_{1}$ and $C_{2}$ are disjoint, one has (see \cite[p. 235]{Ha})
\begin{eqnarray*}
3={\rm deg}(X)={\rm deg}(C_{1}){\rm deg}(C_{ 2})
\end{eqnarray*}
which means that at least one of these curves is a line and then $X$ is a cone.   Let us suppose that $C_{1}$ and $C_{2}$ are not disjoint and have degree at least two. The two planes containing $C_{1}$ and $C_{2}$ must share exactly one point. Otherwise, $X$ coincides with the $\mathbb P^{3}$ spanned by them. We denote by $p$ the intersection point of $C_{1}$ and $C_{2}$. Now we proceed with the same argument of \cite[p. 236 \textit{Calculation II}]{Ha}. If $\Gamma\subset \mathbb P^{4}$ is a general line, we may describe the intersection $\Gamma \cap X$ by considering a general projection $\pi_{\Gamma}:\mathbb P^{4} \dashrightarrow \mathbb P^{2}$ from $\Gamma$. Let $\tilde{C_{i}}\subset \mathbb P^{2}$ be the image of $C_{i}$ by $\pi_{\Gamma}$, $i=1,2$, and $q=\pi_{\Gamma}(p)$. The  points of $\Gamma \cap X$ correspond to the points of $\tilde{C_{1}}\cap \tilde{C_{2}}$ distinct of $q$. We note that the intersection outside $q$ is transverse, thus 
\begin{eqnarray}\label{equality}
3={\rm deg}(X)={\rm deg}(C_{1}){\rm deg}(C_{ 2})-I_{q}
\end{eqnarray}
where $I_{q}=I(\tilde{C}_{1},\tilde{C}_{2};q)$ denotes the intersection multiplicity at $q$. Since $C_{1}$ and $C_{2}$ do not share a tangent line, for a general choice of $\Gamma$ we may assume the same for $\tilde{C}_{1}$ and $\tilde{C}_{2}$. Then $I_{q}$ is the product of the algebraic multiplicity of $\tilde{C}_{1}$ and $\tilde{C}_{2}$ at $q$, say $I_{q}=a_{1}a_{2}$. By the inequality $$a_{i}\le {\rm deg}(C_{i})-1$$ and from (\ref{equality}) one obtains $${\rm deg}(C_{1})+{\rm deg}(C_{2})\le 4.$$
Hence $X$ is the join between the conics $C_1$ and $C_2$. The reader can check that $X$ is uniquely determined up to a projective transformation. We conclude that $X\simeq H\cap S(\mathcal V)$ with $H\in \mathcal V^*\setminus S(\mathcal V)^*$ and this concludes the case where $|\Delta|$ has two irreducible components.


Let us assume that  $|\Delta|=C$ is an irreducible curve.  If $L$ is a general line belonging to $B_{\gamma}$, we have two possibilities:
\begin{enumerate}
\item $\Delta_{L}=2p$;
\item $\Delta_{L}=p+q$, with $p\neq q$.  
\end{enumerate}
From Lemma \ref{2p},  the first case cannot happen because we are assuming that $X$ has nonvanishing Hessian.  If we are in case $(2)$, then a general line of $B_{\gamma}$ is secant to the non-degenerate curve $C$ and then $X=S(C)$. By Lemma \ref{secQRN}, $X=S(C)$ where $C$ is a rational normal curve. This corresponds to the case where $X\simeq H\cap S(\mathcal V)$ where $H\in (\mathbb P^5)^*\setminus \mathcal V^*$. This finishes the proof of Theorem \ref{mainthm}.

\section{The Lefschetz locus in $\GOR(1,5,5,1)$}\label{LL}

\subsection{Artinian Gorenstein algebras and the Lefschetz property}

Let $A=\displaystyle \bigoplus_{i=0}^dA_i,$ be  a graded Artinian  $\mathbb{K}$-algebra with $A_d\neq 0$, we say that $A$ is  
 standard graded  if $A_0=\mathbb K$ and $A$ is generated by $A_1$ as algebra.
 The integer $d$ is called the {\it socle degree} of $A$. 
The {\it codimension} of $A$ coincides with its embedding dimension, that is $\dim A_1$. If  $A=\K[X_0,\ldots,X_N]/I$ is a standard graded Artinian  $\mathbb K$-algebra, where $I$ is an ideal  with $I_1=0$, then $\operatorname{codim} A = N+1$. The {\it Hilbert vector} of $A$ is $\Hilb(A)=(1,a_1,\ldots,a_d)$, where  $a_k = \dim_{\mathbb K} A_k$.

It is a well known fact that $A$ is a Gorenstein algebra if and only if $\dim_{\K} A_d = 1$ and the restrictions of the multiplication in $A$ to complementary degrees $A_k \times A_{d-k} \to A_d$ is a perfect pairing, see \cite[Theorem 2.79]{H-W}. 
For standard graded Artinian Gorenstein algebras, the Hilbert vector is symmetric: $a_i=a_{d-i}$. 

We say that a standard graded Artinian Gorenstein algebra $A$ has the {\em Strong Lefschetz Property}, or simply $A$ has the SLP,  if there exists a linear form $L\in A_1$ such that
$$\bullet L^{d-2i}: A_i \to A_{d-i}$$ is an isomorphism for every $0\leq i\leq \lfloor\frac{d}{2}\rfloor$. In this case  $L$ is called a {\it strong Lefschetz element}. 

The following is the model for standard graded Artinian Gorenstein algebras. 
Let $R=\K[x_0,\ldots,x_N]$ be the polynomial ring in $N+1$ indeterminates and $Q=\K[X_0, ,\ldots,X_N]$ the associated ring of differential operators. Using the classical identification $X_i : = \frac{\partial}{\partial x_i}$, the ring $R$ has a natural structure of  $Q-$module. In fact, differentiation induces a natural action $Q \times R \to R$, given by $(\alpha, f) \mapsto \alpha(f)$. Let $f \in R_d = \K[x_0,\ldots,x_N]_d$ be a homogeneous polynomial of degree $\deg(f)=d \geq 1$.
We define the {\it annihilator ideal} by
$$\Ann_f = \{\alpha \in Q\;:\; \alpha(f)=0\}\subset Q.$$
The homogeneous ideal $\Ann_f$ of $Q$ is also called the  Macaulay dual of $f$. We define $$A_f=\frac{Q}{\Ann_f}.$$
One can verify that $A_f$ is a standard graded Artinian Gorenstein $\K$-algebra of socle degree $d$, see \cite[Section 1.2]{MW}. 
We assume, without loss of generality, that $(\Ann_f)_1=0$. This is equivalent to say that the partial derivatives
 of $f$ are linearly independent, which means that $X = V(f) $ is not a cone.  

By the theory of inverse systems, we get the following characterization of standard graded Artinian Gorenstein $\K$-algebras. It is also called Macaulay-Matlis duality. For a more general discussion of Macaulay-Matlis duality see \cite[Section 2.2]{H-W}, \cite[Section 1.1]{IK}, \cite[Chapter 10]{BS} and \cite[Chapter 21]{E}.

\begin{thm}{\bf \ ( Double annihilator theorem of Macaulay)} \label{G=ANNF} \\
Let $I$ be an ideal of $Q$ such that $Q/I$  is a  standard graded Artinian  $\mathbb K$-algebra of socle degree $d$. Then
$Q/I$ is Gorenstein if and only if there exists $f\in R_d$
such that $I=\Ann_f$.
\end{thm}



 Let $f\in R_d$ be a homogeneous  polynomial  and $A_f=\frac{Q}{\Ann_f}=\displaystyle \bigoplus_{i=0}^dA_i$ the  standard graded Artinian Gorenstein  algebra associated to $f$.
 
Let $\{\alpha_1, \ldots, \alpha_s\}$ be an ordered $\K$-basis of $A_k$, with $k \leq \frac{d}{2}$. The {\it $k$-th Hessian} of $f$ is the matrix 
$$\Hess_f^k = \left[\alpha_i(\alpha_j(f))\right]_{1\le i,j\le s}.$$
Its determinant will be denoted $\hess^k_f$. \neon{Note that  $\Hess^1_f$   is the classical Hessian $\Hess_f$.} 

The following theorem yields a connection between Lefschetz properties and higher Hessians.

\begin{thm}\label{WHessian} \cite{MW} 
Consider $A_f$, where $f\in R$ is a homogeneous polynomial.
An element $L = a_0X_0+\ldots+a_NX_N\in A_1$ is a strong Lefschetz element of $A_f$  if and only if 
$$\hess^k_f(a_0,\ldots, a_N)\neq 0$$ for all $0\le k \le \lfloor d/2 \rfloor$.
\end{thm}

Next we discuss the SLP for  standard graded Artinian Gorenstein algebras  of socle degree $d=3$. For such an algebra the Hilbert vector is 
$$\Hilb(A) = (1,N+1,N+1,1).$$
We denote by $\GOR(1,N+1,N+1,1)$ the space of  standard graded Artinian Gorenstein algebras of socle degree $3$.
By  Theorem \ref{WHessian},  $A_f\in \GOR(1,N+1,N+1,1)$ has the SLP if and only if $\hess_f\neq 0$.

 By \neon{the} Gordan-Noether Theorem, if $N\leq 3$, then $\hess_f=0$ if and only if $X=V(f)$ is a cone (see \cite{GN}). Therefore,
every standard  graded Artinian Gorenstein algebra of socle degree $3$ and codimension $\leq 4$ has the SLP.
When $A$ has codimension $5$, Proposition \ref{classhess} yields the following result.

 

\begin{prop}\label{perazzoalg}
 Let $A$ be a standard graded Artinian Gorenstein $\K$-algebra of Hilbert vector $\Hilb(A) = (1,5,5,1)$.
 Assume that $A$ does not satisfy the SLP. Then $A$ is isomorphic to the following algebra
 $$\frac{\mathbb K [X_0,X_1,X_2,X_3,X_4]}{\left((X_0,X_1,X_2)^2,X_0X_4,X_2X_3,X_1X_3-X_2X_4,X_0X_3-X_1X_4,(X_3,X_4)^3\right)}.$$
\end{prop}

\begin{proof}
The algebra is of the form   $A_f = \K [X_0,X_1,X_2,X_3,X_4]/\Ann_f$, for some  
$f \in \K [x_0,x_1,x_2,x_3,x_4]_3$, not a cone. Theorem \ref{WHessian} implies that SLP fails if and only if $\hess_f=0$. By Proposition \ref{classhess}, we can assume that $f$ has equation $f = x_0x_3^2+x_1x_3x_4+x_2x_4^2$. The desired isomorphism can be obtained using this explicit equation. 
\end{proof}

\subsection{Jordan types}

Let $A=\displaystyle \bigoplus_{i=0}^dA_i$ be a standard graded Artinian $\K$-algebra. For $l \in A_1$ consider the map $\mu_l:A \to A$ given by $\mu_l(x)=lx$.
Since $l^{d+1}=0$,  $\mu_l$ is a nilpotent $\K$-linear map. The Jordan decomposition of such a map is given by Jordan blocks with $0$ in the diagonal, therefore it induces 
a partition of $\dim_{\K}A$ which we denote $\mathcal{J}_{A,l}$. Indeed,  
 the nilpotent linear map $\mu_l:A \to A$ induces  a direct sum decomposition of $A$ into cyclic $\mu_l$-invariant subspaces   $A = \displaystyle \bigoplus_{i=0}^m C_i$ . 
 The partition $\mathcal{J}_{A,l}$ is given by the length $k_i=\dim_{\K}C_i$.
Without loss of generality we consider the partition in a non-increasing order. 

Given a partition $P = p_1\oplus \ldots\oplus p_s$ of $\dim_{\K}A$ with $p_1 \geq \ldots \geq p_s $, we denote $P^ {\vee}$ the {\it dual partition} obtained from $P$ exchanging rows and columns in the Ferrer diagram (diagram of dots). If $P' = p'_1 \oplus \ldots \oplus p'_t$ is another partition of $\dim_{\K}A$ with $p'_1 \geq \ldots \geq p'_t $, we will write $P \preceq P'$ and say that  $P$ is less than $P'$ in the dominance order if for all $k$ we get 
$$p_1+\ldots+p_k \leq p'_1+\ldots+p'_k. $$
If the partition $P$ has repeated terms, say $f_1, f_2, \ldots, f_r$ with multiplicity $e_1, e_2, \ldots, e_r$ respectively,  we write $$P = f_1^{e_1} \oplus \ldots \oplus f_r^{e_r}.$$ 

Since $\K$ is a field of characteristic zero, there is a non empty Zariski open  subset of $\mathcal{U} \subset A_1$ where $\mathcal{J}_{A,l}$ is constant for $l \in \mathcal{U}$, we call it the {\it Jordan type} of $A$ and we denote it $\mathcal{J}_A$. 



The following proposition is a special case of \cite[Proposition 3.64]{H-W}. It shows that SLP can be described by the Jordan type of $A$.

\begin{prop} Suppose that $A=\displaystyle \bigoplus_{i=0}^dA_i$ is a standard graded Artinian $\K$-algebra with $A_d \neq 0$. Then $A$ has the SLP, if and only if, $\mathcal{J}_{A} = \Hilb(A)^{\vee}$.
\end{prop}

Since the generic AG algebra  in $\GOR(1,5,5,1)$ satisfies the SLP, then the generic Jordan type  is $(1,5,5,1)^{\vee}=4^1\oplus 2^4$. 

\begin{ex}\label{example1}\rm

 The algebra $A = Q/\Ann_f$, $f= x_0x_3^2+x_1x_3x_4+x_2x_4^2$, has Hilbert vector $\Hilb(A) = (1,5,5,1)$ and Jordan type 
 $$
 \J_A = 4^1\oplus 2^3\oplus 1^2 \prec 4^1 \oplus 2^4.  $$ 

\end{ex}

The following proposition is a consequence of Proposition \ref{perazzoalg} and Example \ref{example1}.

\begin{prop}\label{jordantype}
 Let $A$ be a standard graded Artinian Gorenstein $\K$-algebra of Hilbert vector $\Hilb(A) = (1,N+1,N+1,1)$. If $N \leq 3$, then $A$ has the SLP. If $N=4$, then the possible Jordan types of $A$ are:
  either $\J_A = 4^1 \oplus2^{4}$ if $A$ has the SLP or $\J_A = 4^1 \oplus2^{3} \oplus 1^2$ if the SLP fails.
\end{prop}

\subsection{The Lefschetz locus in $\GOR(1,5,5,1)$}

The affine scheme $\Gor(T)$ parametrizing AG algebras with Hilbert vector $T$ was described by A. Iarrobino and V. Kanev and a great account of their work can be found in \cite{IK}. 
In their context, $\Gor(T)$ stands for the affine cone of the projective variety denoted $\GOR(T)$ here. As we have seen, by Macaulay-Matlis duality the scheme $\GOR(1,N+1,N+1,1)$ can be identified with the parameter space of degree $3$ homogeneous polynomials $f \in \K[x_0,\ldots,x_N]$, up to scalars, such that $A_f$ has Hilbert vector $\Hilb(A_f) = (1,N+1,N+1,1)$. \neon{As we noted in the paragraph just before Theorem \ref{G=ANNF}, to ensure  $\Hilb(A_f) = (1,N+1,N+1,1)$ we assume that  $({\Ann_f})_1=0$, i.e. that $f$ in not a cone}. Therefore,  we have an identification $$\GOR(1,N+1,N+1,1) \simeq \P^{\nu(N)} \setminus \CC_N$$  
where $\nu(N) = \binom{N+3}{3}-1$ and $\CC_N$ is the parameter space of cubic cones in $\P^{N}$. In particular
$$\GOR(1,5,5,1) \simeq \P^{34} \setminus \CC_4.$$


By Theorem \ref{WHessian}, an AG algebra $A$ of socle degree $3$ has the SLP if and only if its dual generator $f$ satisfies $\hess_f \neq 0$. 
 Proposition \ref{jordantype} gives  a description of their Jordan types. We summarize this discussion in the following proposition.

\begin{prop}\label{GorN+1}
 The space $\GOR(1,N+1,N+1,1)$  can be identified with $\P^{\nu(N)} \setminus \CC_N$,
 where $\nu = \binom{N+3}{3}$ and $\CC_N$ is the  space of cubic cones in $\P^{N}$.
 For $N \leq 3$ all the algebras in $\GOR(1,N+1,N+1,1)$ have the SLP. For $N=4$, the locus in $\GOR(1,5,5,1)$ of algebras 
 satisfying SLP is $\GOR(1,5,5,1)\setminus Y$ where $Y$ can be identified with the locus formed by $f\in\P^{\nu(N)} \setminus \CC_N$ with vanishing Hessian.
\end{prop}

\section{Parameter spaces}\label{PS}
\neon{
In this section we find  parameter spaces for the locus of  cubic cones  in $\P^4$, for the locus of cubics with vanishing Hessian and for their intersection; 
using these descriptions, we compute their dimensions and degrees.

To compute the degree our principal tools are the Segre and Chern classes  of a vector bundle. 
We refer the reader to \cite[\S3.1,\,p.\,47]{FUL}   or \cite[\S10.1 and Ch 5]{EH}  for a systematic treatment on Segre and Chern classes.

  Throughout this section  we denote by $V$ the vector space $\K^5$. Hence we have: $\P^4=\P(V)$, $R=\s{}V^* $, $\P^{34}=\P(R_3)=\P(\s 3(V^*))$.}

\subsection{Parameter space for cubic cones in $\P^4$}\label{cones}

In this section we find a parameter space for cubic cones  in $\P^4$ and compute its dimension and degree. 

 A cubic cone in $\P^4$ is determined by a point  $x\in \P^4$ and a cubic hypersurface in the $\P^3$ projectivization of the quotient $V/x$. 
Consider the tautological sequence on $\P^4=\P(V)$:
\begin{equation}\label{tauto1}
0 \rightarrow\OX_{\P^4}(-1)\rightarrow  \mathcal O_{\P^4}\otimes V \rightarrow \mathcal{P}\rightarrow 0.
\end{equation} 

The fiber of $\P(\mathcal{P})$ over $x\in \P^4$ can be identified   with    $\P^3=\P(V/x)$.    Therefore $\P(\s 3(\mathcal{P}^{*}))$ parametrizes  cubic hypersurfaces lying in each $\P^3$, where $\mathcal{P}^{*}$ denotes the dual  of the vector bundle  $\mathcal{P}$.  
Note that $\Fe=\s 3(\mathcal{P}^{*})$ is a subbundle of $\mathcal O_{\P^4}\otimes \s 3(V^{*})$. Thus we have two projections 

\[
\xymatrix 
{
&~
\P(\Fe) 
\ar[dl]_{p_{1}} \ar[dr]^{p_{2}}&  \\
\P^4 & &\P^{34}=\P(\s 3(V^{*}))}
\]
where $p_{2}$ is generically injective and $\CC_4$ is the image of $p_{2}$. 
We have the following result. 

\begin{prop}\label{cone4}
Let $\CC_4\subset \mathbb P^{34}$ be the space of cubic cones in $\P^4$. 
Then 
the dimension  of $\CC_4$ is $23$;  its degree  is given by the Segre class $s_{4}(\Fe)$ and is equal to
1365.


\end{prop}
\proof
 The dimension can be computed by
 $$\dim(\CC_4)=\dim(\P(\Fe))=4+rk(\Fe)-1$$
 and $rk(\Fe)=20$.

To compute  the degree, write  $H$ for the
hyperplane class of $\P^{34}$.  We have
$p_{2}^*H=c_1\mathcal{O}_{\Fe}(1)=:h$. We may compute
$$
\deg\CC_4=\int_{\P^{34}}H^{23}\cap\CC_4=\int_{\P(\Fe)}
h^{23}=\int_{\P^4} p_{1*}(h^{23})=\int_{\P^4}
s_{4}(\Fe).
$$
Using Macaulay2 \cite{GS} we find $s_{4}(\Fe)\cap \P^4=1365$  (see the Scripts in \neon{\S}\ref{scripts}).
\endproof

\neon{This calculation  can be generalized to determine the degree of the locus $\CC_{d,n}$ of cones in the space parametrizing hypersurfaces of degree $d$ in $\P^n$. Using another characterization for cones, this degree is calculated in \cite[Proposition 7.8, p. 257]{EH}. There the authors obtain the following formula for the degree of the locus in $\P^{\binom{n+d}{n}}$ of cones of degree $d$:  $$\deg(\CC_{d,n})=\binom{\binom{n+d-1}{n}}{n}.$$}

%
%
%
%
%
%
%
%
%
%
%

\subsection{Parameter space for cubics with vanishing Hessian}\label{nonormal}

This section is devoted to the description of  the parameter space for  cubic hypersurfaces with vanishing Hessian in $\P^4=\P(V)$.  
 
Let $\mathcal{H}\subset \P^{34}$ be the locus of hypersurfaces with vanishing Hessian. 
As in Proposition \ref{GorN+1} we denote  $Y=\mathcal{H}\setminus \mathcal{C}_4$.

On the other hand, let  $\mathcal K\subset \P^{34}$ be the locus formed by $f\in \P^{34}$ such that  $f\in I^2_W$, where $I_W$ is the ideal of  some  $2$-plane $W\subset \P^4$. It is an irreducible subvariety of $\P^{34}$, actually we will show below that it is the image by a morphism of a projective bundle over $\mathbb G(2,4)$.
\begin{rmk}\rm  We claim that $\overline{Y}  = \mathcal K$. 
Note that $\mathcal K \subset \overline{Y}$ by Lemma \ref{idealquadratico}.  
Reciprocally  Proposition \ref{classhess} yields  $Y\subset \mathcal K$. This concludes the claim.  
\end{rmk}

We shall find a parameter space for $\overline{Y}$ using the above characterization. 
 Let $\grass=\grass(2,4)$ denote the Grassmannian of $2$-planes in $\P^4$.
We have the following tautological sequence on  $\grass$:
\begin{equation}\label{tauto}
0 \rightarrow\T\rightarrow \mathcal O_{\grass}\otimes V \rightarrow Q\rightarrow 0
\end{equation} 
where $\T$ is a subbundle of rank $3$ and $Q$ is a bundle of rank $2$. 


Consider  the multiplication map 
$$\varphi:\s 2Q^{*}\otimes V^{*}\to  \mathcal O_{\grass}\otimes \s 3(V^{*}).$$
It defines a map of vector bundles whose image parametrizes the set of pairs $(W,f)\in \grass\times \s 3(V^{*})$ such that $f\in I_{W}^2$. 

We claim that the kernel of $\varphi$ is exactly 
$\wedge^2Q^{*}\otimes Q^{*}$. Therefore we have an exact sequence  
\begin{equation}\label{varphi}\xymatrix
{
0 \rightarrow \wedge^2Q^{*}\otimes Q^{*}\rightarrow \s 2Q^{*}\otimes V^{*}\xrightarrow[]{\neon{\varphi}} \Ee\rightarrow 0}
\end{equation} 
where $\Ee={\rm Im}\varphi$.  

Let us prove the claim. Consider the following exact diagram

\begin{equation}\label{N}\xymatrix
{\wedge^2Q^{*}\otimes Q^{*}\ar@{>->}[r]\ar@{>->}[d]& Ker\varphi \ar@{->>}[r] \ar@{>->}[d]&  
Ker\bar{\varphi}\ar@{>->}[d]\\
\s 2Q^{*}\otimes Q^{*} \ \,
\ar@{>->}[r] \ar@{->>}[d]^{m}&
\s 2Q^{*}\otimes V^{*}
\ar@{->>}[r]\ar@{->>}[d]_\varphi
& \s 2Q^{*}\otimes \T^{*}
\ar@{->>}[d]^{\bar{\varphi}}
 \\
\s 3Q^{*}\ \,
\ar@{>->}[r]& Im\varphi
\ar@{->>}[r]&  \,Im\bar{\varphi},
}
\end{equation}

 To prove the claim it  suffices to prove that $\bar{\varphi}$ is injective. \neon{Let $W\in \grass$ be a two-plane,  
as $\grass$ is a homogeneous variety, we can assume that $W=V(x_{3},x_{4})$}. Assume  $\bar{f}$ lies in the fiber of $\s 2Q^{*}\otimes \T^{*}$ over $W$, then it can be written as 
\[
\bar{f}=x_{3}^2\otimes f_{0}+x_{3}x_{4}\otimes f_{1}+x_{4}^2\otimes f_{2}
\]
where $f_{i}$ are homogeneous polynomials of degree one, $f_{i}\in \langle x_{0},x_{1},x_{2}\rangle$.
Hence  $\bar{\varphi}(\bar{f})=0$ means that $\varphi(f)$ lies in the fiber of $\s 3Q^*$, i.e.  
$x_{3}^2 f_{0}+x_{3}x_{4} f_{1}+x_{4}^2 f_{2}\in \s 3(x_3,x_4)$, therefore $f_0=f_1=f_2=0$.

We note that $\Ee$ corresponds to a vector bundle over $\grass$ whose projectivization coincides with  the 
incidence variety $$\P(\Ee)=\{(W,f)\in\grass\times \P^{34}\;:\; f\in I_{W}^2\}.$$ 
Let us denote by   $p_1:\P(\Ee) \longrightarrow \grass$ and  $p_2:\P(\Ee) \longrightarrow \P^{34}$ the natural projections. 
We see that $\mathcal K$ is the image of $p_{2}$ and $p_{2}$ is generically injective.
Hence one obtains the following result.

\begin{prop}\label{degI2}
The locus $\mathcal K$ is the birational image of a projective bundle over the Grassmannian $\mathbb{G}$. The  dimension of $\mathcal K$ is $18$ and its degree  is $29960$, the degree of the Segre class $s_{6}(\Ee)$ . 
\end{prop}
\proof
 The dimension can be computed as 
 \[
 \dim(\mathcal K)=\dim(\P(\Ee))=\dim\grass+rk(\Ee)-1.
 \] 
Since  $rk(\Ee)=13$ and  $\dim\grass=6$ the result follows.

We shall  write  $H$ for the
hyperplane class of $\P^{34}$.  We have
$p_{2}^*H=c_1\mathcal{O}_{\Ee}(1)=:h$. \neon{Using 
\cite[\S3.1,\,p.\,47]{FUL} } the degree of $\mathcal K$ is given by 
$$
\int_{\P(\Ee)}
h^{18}=\int_\grass p_{1*}(h^{18})=\int_\grass{}
s_{6}(\Ee).
$$
By sequence (\ref{varphi}), using the properties of Chern classes \cite[\S3.2,\,p.\,50]{FUL}, we have 
$$s(\Ee)=s( \s 2Q^{*}\otimes V^{*})c(\wedge^2Q^{*}\otimes Q^{*})=s( \s 2Q^{*})^5c(\wedge^2Q^{*}\otimes Q^{*}).$$ 
We can compute these characteristic classes using Macaulay2:  $s_{6}(\Ee)=29960$ (see the Scripts in \S\ref{scripts}).

\endproof

%
%

\subsection{The locus $\mathcal K \cap \CC_4$}

In this section we describe the intersection $\mathcal K\cap \CC_4$.  Let us consider again the  tautological sequence  on $\grass=\grass(2,4)$:
\begin{equation*}
0 \rightarrow\T\rightarrow \mathcal O_{\grass}\otimes V\rightarrow Q\rightarrow 0.
\end{equation*} 
%
Note that $\T$ corresponds to a vector bundle whose projectivization coincides with the  incidence variety
$$\P(\T)=\{(W,p)\in\grass\times \P^4\;:\; p\in W\} .$$
Let   $q_1:\grass\times\P^4 \longrightarrow \grass$ and  $q_2:\grass\times\P^4 \longrightarrow \P^{4}$  be the natural projections and denote by  $\pi_1=q_1|_{\P(\T)}$ and $\pi_2=q_2|_{\P(\T)}$ their restrictions to $\P(\T)$. 
%

We will construct  a vector bundle $\Ee_1$ over $\P(\T)$ with a birational morphism $\P(\Ee_1) \longrightarrow \mathcal K\cap\CC_4$.  Consider the tautological sequence on $\P^4=\P(V)$:
\begin{equation*}
0 \rightarrow\OX_{\P^4}(-1)\rightarrow  \mathcal O_{\P^4}\otimes V\rightarrow \mathcal{P}\rightarrow 0. 
\end{equation*} 
Over $\P(\T)$ we have the following multiplication map:
$$\xi:\pi_{1}^*\s 2Q^{*}\otimes \pi_{2}^*\mathcal{P}^{*}\to   \mathcal O_{\P(\T)}\otimes\s 3(V^{*}).$$
Its kernel is exactly  $\pi_{1}^*(\wedge^2Q^{*}\otimes Q^{*})$. This  can be proved using the same argument we applied to the multiplication map $\varphi$ that appears in Section \ref{nonormal}. 
Therefore we have an exact sequence  
 \begin{equation}\label{varphi2}
0 \rightarrow \pi_{1}^*(\wedge^2Q^{*}\otimes Q^{*})\rightarrow \s 2\pi_{1}^*Q^{*}\otimes \pi_{2}^*\mathcal{P}^{*}\rightarrow \Ee_{1}\rightarrow 0
\end{equation} 
where $\Ee_1={\rm Im}\xi$. Note that $\Ee_1$ defines a vector bundle whose fiber over a point $(W,p)\in \P(\T)$ is equal to  the vector space
$$\{f\in  \s 3(V^{*})\;:\; f\,\text{is a cone with vertex }\, p\,\text{and}\, f\in I_{W}^2\}.$$

Let us consider the natural projections  $p_1:\P(\Ee_1)\longrightarrow \P(\T)$ and  $p_2:\P(\Ee_1)\longrightarrow \P^{34}$.
%
%
We see that $p_{2}$ is generically  injective and its image is  $\mathcal K\cap \CC_4$. Hence one obtains the following result. 
\begin{prop}\label{degInt} $\empty$
\begin{enumerate}
\item $\mathcal K\cap \CC_4$ is a  divisor in $\mathcal K$.
\item The degree of $\mathcal K\cap \CC_4$ is $116420$. It is determined by $$3s_{6}(\Ee)+(c_{1}(\s 2 Q^{*})+c_{1}(Q))s_{5}(\Ee)\cap [\grass]$$
 where $\Ee$ is the vector bundle of Section \ref{nonormal}. 
\end{enumerate}
\end{prop}
\proof 
The dimension of $\mathcal K\cap \CC_4$  coincides with
\[
\dim(\P(\Ee_1))=\dim\P(\T)+rk(\Ee_1)-1. 
\] 
Since $rk(\Ee_1)=10$ and $\dim\P(\T)=8$ the result follows. In order to compute  the degree, write  $H$ for the
hyperplane class of $\P^{34}$.  We have
$p_{2}^*H=c_1\mathcal{O}_{\Ee_1}(1)=:h$. The degree of $\mathcal K\cap \CC_4$ is given by
$$
\int_{\P(\Ee_1)}
h^{17}=\int_{\P(\T)}  p_{1*}(h^{17})=\int_{\P(\T)}
s_{8}(\Ee_1).
$$

We claim that the following identity occurs in the Chow ring of $\grass\times \P^4$: 
 \begin{eqnarray}\label{ideT}
 [\P(\T)]=c_{2}(q_2^*\OX_{\P^4}(1)\otimes q_1^*Q)\cap [\grass\times \P^4].
 \end{eqnarray}
 Indeed, exact sequences (\ref{tauto1}) and (\ref{tauto}) yield  a map $\theta$ :
\[
\xymatrix 
{
&~
q_2^*\OX_{\P^4}(-1) 
\ar[dr]^{\theta} \ar[d]&  \\
q_1^*\T\ar[r]&\ar[r] \mathcal O_{\grass\times \P^4}\otimes V &q_1^*Q&
}
\]
which has  $\P(\T)$ as zeros. Then it induces a regular section 
$\sigma:\grass\times \P^4\to q_2^*\OX_{\P^4}(1) \otimes q_1^*Q$ which has $\P(\T)$ as zeros. This proves identity (\ref{ideT}) (see \cite[ex 3.2.16,\,p.\,61]{FUL}).

 Putting together sequences (\ref{varphi}) and  (\ref{varphi2}) we get 
\[
s(\Ee_1)=s(\pi_{1}^*\Ee)c(\mathcal G).
\]
%
Where  $\mathcal G=\s 2\pi_{1}^*Q^{*}\otimes \pi_{2}^*\OX_{\P^4}(1)$.
Using this, we deduce: 
$$s_{8}(\Ee_1)=s_{8}(\pi_{1}^*\Ee)+s_{7}(\pi_{1}^*\Ee)c_{1}(\mathcal G)+
s_{6}(\pi_{1}^*\Ee)c_{2}(\mathcal G)+s_{5}(\pi_{1}^*\Ee)c_{3}(\mathcal G).$$
We note that  $s_{i}(\pi_{1}^*\Ee)=0$ for $i>6$ because  $\Ee$ is a vector bundle over  $\grass$ which has dimension $6$. From this and (\ref{ideT})  one obtains  that $s_{8}(\Ee_1)\cap [\P(\T)]$ coincides with 
\begin{eqnarray*}
\left(s_{6}(\pi_1^*\Ee)c_{2}(\mathcal G)+s_{5}(\pi_1^*\Ee)c_{3}(\mathcal G)\right)\cap (k^2+kc_{1}(\pi_1^*Q)+c_{2}(\pi_1^*Q))\cap [\grass\times \P^4]
\end{eqnarray*}
where $k=c_{1}(q_2^*\OX_{\P^4}(1))$. 

In what follows we will omit the pull-back. Computing the Chern classes of a tensor product, we obtain 
\begin{eqnarray}\label{TP}
\left\{ \begin{array}{ll}
c_{2}(\mathcal G)=c_{2}(\s 2 Q^{*})+2c_{1}(\s 2 Q^{*})k+3k^2\\
c_{3}(\mathcal G)=c_{3}(\s 2Q^{*})+c_{2}(\s 2Q^{*})k+c_{1}(\s 2Q^{*})k^2+k^3.
\end{array} \right.
\end{eqnarray}
%
Observe that  \[
c_{i}(Q)s_{j}(\Ee)=c_{i}(\s 2 Q^{*})s_{j}(\Ee)=0, \;\; i+j>6.
\]  Using this and (\ref{TP}), we can reduce the computation of $s_{8}(\Ee_1)\cap [\P(\T)]$  to:
\begin{align*}
(s_{6}(\Ee)3k^2+s_{5}(\Ee)(c_{1}(\s 2Q^{*})k^2+k^3))\cap k(k+c_{1}(Q))\cap [\grass\times \P^4].
\end{align*}


Finally, since $k^5=0$ and $\grass$ has dimension $6$ we have 
\begin{align*}
s_{8}(\Ee_1)\cap [\P(\T)]=
3s_{6}(\Ee)+&(c_{1}(\s 2Q^{*})+c_{1}(Q))s_{5}(\Ee)\cap[\grass].
\end{align*} 

This number can be computed using Macaulay2/Schubert2, we find $116420$. This concludes the proof of  proposition.
\endproof

Now the  Theorem \ref{mainthm2} of the introduction is \neon{a} consequence of Propositions \ref{GorN+1}, \ref{cone4}, \ref{degI2} and \ref{degInt}.

\subsection{Scripts.}\label{scripts}

\label{Scripts for Macaulay2}

\label{script}
{\small

\begin{verbatim}
loadPackage "Schubert2"
G=flagBundle({1,4})
-- Grassmannian of lines  in 5-space
(S,Q)=G.Bundles
-- names the sub and quotient bundles on G
R=dual (Q)
F=symmetricPower(3,R)
--Computes the classes in Proposition 4.1:
integral(segre(4,F))
\end{verbatim}
}

\label{script}
{\small
\begin{verbatim}

loadPackage "Schubert2"
G=flagBundle({3,2})
-- Grassmannian of 3-planes  in 5-space
(S,Q)=G.Bundles
R=dual (Q)
A=symmetricPower(2,R)
B=A^5
C=exteriorPower(2,R)*R
E=B-C
--Computes the classes in Proposition 4.3:
integral(segre(6,E))

--Computes the classes in Proposition 4.4:
integral(3*segre(6,E)+(chern(1,A)+chern(1,Q))*segre(5,E))
\end{verbatim}
}
%

{\bf Acknowledgments}.
We wish to thank F. Russo, I. Vainsencher,  G. Staglian\`o and N. Medeiros for pointing out some mistakes in an early draft and for inspiring conversations on the subject.

\end{document}